\documentclass[12pt]{amsart}
\usepackage{amsmath}
\usepackage{float}
\allowdisplaybreaks[4]

\newtheorem{theorem}{Theorem}[section]

\newtheorem{proposition}[theorem]{Proposition}
\newtheorem{corollary}[theorem]{Corollary}

 \theoremstyle{definition}
\newtheorem{definition}[theorem]{Definition}

\newtheorem{condition}{Condition}
\theoremstyle{remark}
\newtheorem{remark}[theorem]{Remark}

\numberwithin{equation}{section}

\begin{document}

\title[Neumann problem]
{The Neumann problem for a class of Hessian quotient type equations }

\author{Jiabao Gong}
\address{Faculty of Mathematics and Statistics, Hubei Key Laboratory of Applied Mathematics, Hubei University,  Wuhan 430062, P.R. China}
\email{202321104011284@stu.hubu.edu.cn}

\author{Zixuan Liu}
\address{Faculty of Mathematics and Statistics, Hubei Key Laboratory of Applied Mathematics, Hubei University,  Wuhan 430062, P.R. China}
\email{202221104011303@stu.hubu.edu.cn}

\author{Qiang Tu$^{\ast}$}
\address{Faculty of Mathematics and Statistics, Hubei Key Laboratory of Applied Mathematics, Hubei University,  Wuhan 430062, P.R. China}
\email{qiangtu@hubu.edu.cn}


\thanks{This research was supported by funds from Natural Science Foundation of Hubei Province, China , No. 2023AFB730 and the National Natural Science Foundation of China No. 12101206.}
\thanks{$\ast$ Corresponding author}

\begin{abstract}
In this paper, we consider the Neumann problem for a class of Hessian quotient equations  involving a gradient term on the right-hand side in Euclidean space.
More precisely, we derive the interior gradient estimates  for the $(\Lambda, k)$-convex solution of Hessian quotient equation $\frac{\sigma_k(\Lambda(D^2 u))}{\sigma_l(\Lambda(D^2 u))}=\psi(x,u,D u)$ with $0\leq l<k\leq C^{p-1}_{n-1}$ under the assumption of the  growth condition.
 As an application,  we  obtain the global a priori estimates  and the existence theorem for the Neumann problem of  this Hessian quotient type equation.
\end{abstract}

\maketitle
\section{Introduction}

Let $\Omega\subset\mathbb{R}^n$ be a domain with $C^4$ boundary, $u$ be a $C^2$ function on $\Omega$ and $\lambda=(\lambda_1,\cdots,\lambda_n)$ be the eigenvalues of the Hessian matrix $D^2u$.
Given an integer $p$ with $1\leq p\leq n-1$, set
$$\mathfrak{J}=\{(i_1,\cdots,i_p)|1\leq i_1<\cdots<i_p\leq n\}.$$
For any $I=(i_1,\cdots,i_p)\in \mathfrak{J}$,  we define
$$
\Lambda_I(D^2 u)=\lambda_{i_1}+\cdots+\lambda_{i_p}.
$$
For convenience, fix an order for the elements in $\mathfrak{J}$:
\begin{eqnarray*}
I_1,\cdots,I_N,~~\mbox{with}~~N=C^p_n=\frac{n!}{p!(n-p)!},
\end{eqnarray*}
and write
$$\Lambda(D^2u)=(\Lambda_{I_1}(D^2u),\cdots,\Lambda_{I_N}(D^2u))=(\Lambda_1(D^2u),\cdots,\Lambda_N(D^2u)).$$
In this paper, we consider the following Hessian quotient type equation
\begin{eqnarray}\label{(1.1)}
\frac{\sigma_k(\Lambda(D^2 u))}{\sigma_l(\Lambda(D^2 u))}=\psi(x,u,D u),\quad x\in \Omega,
\end{eqnarray}
where $\sigma_k$ is the $k$-th elementary symmetric function and $0\leq l<k\leq N$, $\psi$ is a given function. To ensure the ellipticity of equation \eqref{(1.1)}, the admissible set  is defined as follows.
\begin{definition}\label{def-1}
A function $u\in C^2(\Omega)$ is called $(\Lambda, k)$-convex if $\Lambda(D^2 u) \in \tilde{\Gamma}_k$ for any $x\in \Omega$,
where $\tilde{\Gamma}_k$ is the Garding cone
\begin{eqnarray*}\label{cone}
\tilde{\Gamma}_{k}=\{\lambda \in \mathbb{R}^N: \sigma_{j}(\lambda)>0,~ \forall ~ 1\leq j \leq k\}.
\end{eqnarray*}
\end{definition}

If   $p=1$,  equation \eqref{(1.1)} becomes the classic  Hessian quotient equation
$$\frac{\sigma_k(D^2u)}{\sigma_l(D^2 u)}=\psi(x, u,D u),$$
 and the corresponding Dirichlet problem and Neumann problem have been studied extensively in the past decades. When the right-hand side function depends only on $x$, the Dirichlet problem for Hessian equations was considered by  Caffarelli-Nirenberg-Spruck  \cite{CNS85}, where they
treated a general class of fully nonlinear equations under conditions on
the geometry of $\partial \Omega$. Then the same results for the Dirichlet problem of Hessian quotient equations have been established by Trudinger in \cite{Tr95}. Guan \cite{Guan12} solved the Dirichlet problem in a domain with no geometric restrictions to the boundary under essentially optimal structure conditions.
Trudinger has solved the Neumann problem for Hessian equations when the domain is a ball in \cite{TR87}. At the end of his paper, Trudinger conjectured that  a priori estimates and the existence result for the Neumann problem of Hessian equations still  hold in  sufficiently smooth uniformly convex domains. Recently, Ma-Qiu \cite{MQ19} gave a positive answer to this problem and
solved the Neumann problem of k-Hessian equations in strictly convex domains. Chen-Zhang \cite{CZ21} generalized the result to the Neumann problem of Hessian quotient equations.
When the right-hand side function depends  on $x, u$ and $Du$, the interior gradient estimates  of Hessian quotient equations has been established by Chen \cite{Chen15}. Under the convexity assumption of  $\psi$ on $\nabla u$, Guan-Jiao \cite{Guan15, GBH2016} obtained the $C^2$ estimates for  Hessian quotient equations and the corresponding Dirichlet problem has been considered.
In \cite{D}, Dong studied the interior and boundary gradient estimates for solutions to the Neumann problem of Hessian quotient equations on Riemannian manifolds.

If $p=n-1$, equation \eqref{(1.1)} is known as the following  Hessian quotient type equation
\begin{eqnarray}\label{(1.2)}
\frac{\sigma_k(U[u])}{\sigma_l(U[u])}=\psi(x, u, D u),
\end{eqnarray}
where $U[u]=\Delta uI-\nabla^2 u$, which has attracted the attention of many authors due to its geometric applications such as the Gauduchon conjecture in complex geometry. In \cite{Chu20}, Chu-Jiao proved Pogorelov type estimates of solutions to Hessian type equation \eqref{(1.2)} with $l=0$. Chen-Tu-Xiang \cite{CTX2} generalized the result to equation \eqref{(1.2)} in case $0\leq l<k-1$. In \cite{Mei}, Mei established the interior $C^2$ estimates for equation \eqref{(1.2)}  when $k \leq n-1$ and constructed the non-classical solutions when $k=n$. Then the corresponding Dirichlet problem for  equation \eqref{(1.2)} has been studied by Chen-Tu-Xiang \cite{CTX3} on Riemannian manifolds. When the right-hand side function depends only on $x$,
 Chen-Dong-Han \cite{CDH}  derived an interior  priori Hessian estimates for equation \eqref{(1.2)} in case $0\leq l<k<n$. As an application, they also considered  Liouville theorem and  the Neumann problem for equation \eqref{(1.2)}. Dong-wei \cite{DW-22} studied the Neumann problem of Hessian quotient equations in the complex domain.

For $2\leq p \leq n-1$, there were also many research results for equation \eqref{(1.1)}.   When $l=0$  and $k = 1$, equation \eqref{(1.1)} becomes semilinear elliptic partial differential equation, which was studied by Han-Ma-Wu \cite{HMW2011}, and an existence theorem of strictly $p$-convex starshaped hypersurface for the corresponding curvature equation was obtained. If $l= 0$ and $k = N$, equation \eqref{(1.1)} becomes the following $p$-Monge-Amp\`re type  equation
\begin{equation}\label{1.5}
\Pi_{1\leq i_1<\cdots<i_p\leq n} \left(\lambda_{i_1}+\cdots+\lambda_{i_p}\right) =\psi(x, u, Du)
\end{equation}
Dinew \cite{Dinew} studied  equation \eqref{1.5} in case $\psi=\psi(x,u)$, and obtained  first and second order interior estimates for $p$-plurisubharmonic solutions. Chu-Dinew \cite{CD2023}
studied a general class of Hessian elliptic equations, including  $p$-Monge-Amp\`ere equation, and proved Liouville theorem under new additional conditions. The corresponding curvature equations have been studied by Dong in \cite{Dong22, Dong24}.
Recently, Zhou \cite{ZJ} considered the following Hessian quotient type curvature equation
$$\frac{\sigma_k(\Lambda(\kappa))}{\sigma_l(\Lambda(\kappa))}=\psi(X, \nu),$$
where $X=(x, u(x)) \in B_r\times \mathbb{R}$, $\nu$ is
the outer unit normal and $\kappa$ is the principle curvatures of an isometrically immersed hypersurface $M$ in Euclidean space $\mathbb{R}^{n+1}$.
They established interior curvature estimates under the condition of $0\leq l<k\leq C_{n-1}^{p-1}$, and proved Bernstein  theorem for this  curvature type equation.

From analysis point of view, a natural problem is whether we can establish  a priori interior  estimates for  equation \eqref{(1.1)} and apply these estimates to the Neumann problem.

In order to obtain the interior gradient estimates of the solution for equation \eqref{(1.1)}, we introduce the following growth condition which can refer to Dong's works in \cite{D}.

\begin{condition}\label{growth}
There exists  positive constants $C_1$ and $M_1>1$ such that for any $(x, z, p)\in \Omega\times \mathbb{R}\times \mathbb{R}^n$ with $|p|>M_1$, $\tilde\psi:=\psi^{\frac{1}{k-l}}$ satisfies
\begin{eqnarray}\label{(s1)}
|{\tilde\psi}_{x}|+|{\tilde\psi}_z||p|+|{\tilde\psi}_{p}|p^2 \leq C_1|p|^{2+\gamma}.
\end{eqnarray}
for a constant $\gamma < 1$.
\end{condition}

Then  a priori interior gradient estimates for the $(\Lambda, k)$-convex solution of equation \eqref{(1.1)} is as follows.
\begin{theorem}\label{C11}
Let $0\leq l<k\leq C^{p-1}_{n-1}$ and $u\in C^3(B_r(0))$ be the $(\Lambda,k)$-convex solution of equation \eqref{(1.1)} in $B_r(0)$. Suppose that $\tilde{\psi}\in C^1(B_r\times\mathbb{R}\times\mathbb{R}^n)$ with $\tilde{\psi}>0$ satisfy the  \textbf{Condition \ref{growth}}. Then
\begin{eqnarray}\label{(1.2.1)}
|Du(0)|&\leq C \bigg(\frac{[\underset{B_r(0)}{osc} u]}{r} + [\underset {B_r(0)}{osc} u]^{\frac{2}{1-\gamma}}+[\underset{B_r(0)}{osc}u]^{\frac{1}{1-\gamma}} \bigg),
\end{eqnarray}
where $C$ is a positive constant depending on $n, k, l, p, M_1$ and $C_1$.
\end{theorem}

\begin{remark}
When $\psi$ depends only on $x$ and $u$, we can obtain the following estimate
\begin{eqnarray*}
|Du(0)|&\leq C(\frac{\underset{B_r(0)}{osc} u}{r}+\underset{B_r(0)}{osc} u+{[\underset{B_r(0)}{osc} u]}^{\frac{2}{3}})
\end{eqnarray*}
without  \textbf{Condition \ref{growth}}. Here $C$ is a positive constant depending on $k, l, n, p$ and $|\psi|_{C^1}$.
\end{remark}

Next we consider the the Neumann problem of equation \eqref{(1.1)}. Under some assumptions for $\psi$ and $\varphi$, we solve the
existence  for the Neumann problem, based on  a priori estimates for the solution to  equation  \eqref{(1.1)}. The main theorem is as follows.

\begin{theorem}\label{T1}
Let $0\leq l<k\leq C^{p-1}_{n-1}$ and $\Omega\subset\mathbb{R}^n$ be a  domain with $C^4$ boundary. Assuming that $\psi\in C^3(\Omega\times\mathbb{R}\times\mathbb{R}^n)>0$ satisfy the \textbf{Condition \ref{growth}}
and $\varphi\in C^3(\partial \Omega\times \mathbb{R})$.
Suppose that there exists positive constants $c_0, \alpha_0$ such that $-\varphi_u\geq c_0>0$ and  $\tilde{\psi}^{-1}_u\leq -\alpha_0<0$,
then there exists a unique $(\Lambda,k)$-convex solution $u \in C^{3,\alpha}(\bar \Omega) $ for the following Neumann problem
\begin{eqnarray}\label{(1.4)}
\left\{\begin{matrix}
\frac{\sigma_k(\Lambda(D^2 u))}{\sigma_l(\Lambda(D^2 u))}=\psi (x,u,D u) , & in~~ \Omega \\
 u_\nu =\varphi (x,u), & on~~\partial \Omega
\end{matrix}\right.
\end{eqnarray}
where $\nu$ is the unit outer normal vector of $\partial \Omega$.
\end{theorem}

Nevertheless, we have the following corollary.
\begin{corollary}
Let $0\leq l<k\leq C^{p-1}_{n-1}$ and $\Omega\subset\mathbb{R}^n$ be a  domain with $C^4$ boundary. For any positive functions $\psi \in  C^2(\bar\Omega)$ and  $ \varphi \in C^3(\partial \Omega)$, then there exists a unique $(\Lambda,k)$-convex solution $u \in C^{3,\alpha}(\bar \Omega) $ for the following Neumann problem
\begin{eqnarray}\label{(1.8)}
\left\{\begin{matrix}
\frac{\sigma_k(\Lambda(D^2 u))}{\sigma_l(\Lambda(D^2 u))}=\psi (x) , & in~~ \Omega \\
 u_\nu =-\beta u+\varphi (x), & on~~\partial \Omega
\end{matrix}\right.
\end{eqnarray}
where $\beta$ is a positive constant.
\end{corollary}

The rest of this paper is organized as follows. In Section 2, we give some properties of $\Lambda_I(D^2 u)$ . In Section 3, we prove the $C^0$ estimates. In Section 4, we give global gradient estimates. In Section 5, the second order estimates are derived. In Section 6, we prove the existence of the solution.

\section{Preliminaries}
In this section, we recall the definition and some basic properties of elementary symmetric function, which could be found in~\cite{CCQ}.

Let $\lambda=(\lambda_1,\dots,\lambda_n)\in\mathbb{R}^n$, we recall
the definition of elementary symmetric function for $1\leq k\leq n$,
\begin{equation*}
\sigma_k(\lambda)= \sum _{1 \le i_1 < i_2 <\cdots<i_k\leq
n}\lambda_{i_1}\lambda_{i_2}\cdots\lambda_{i_k}.
\end{equation*}
We also set $\sigma_0=1$ and $\sigma_k=0$ for $k>n$ or $k<0$. The Garding cone is defined by
\begin{equation*}
\Gamma_k  = \{ \lambda  \in \mathbb{R}^n :\sigma _i (\lambda ) >
0,~\forall~ 1 \le i \le k\}.
\end{equation*}
We denote $\sigma_{k-1}(\lambda|i)=\frac{\partial
\sigma_k}{\partial \lambda_i}$ and
$\sigma_{k-2}(\lambda|ij)=\frac{\partial^2 \sigma_k}{\partial
\lambda_i\partial \lambda_j}$. Next, we list some properties of
$\sigma_k$ which will be used later.

\begin{proposition}\label{P1}
Let $\lambda=(\lambda_1,\dots,\lambda_n)\in\mathbb{R}^n$ and $1\leq k\leq n$, then we have
\begin{enumerate}
\item[(\romannumeral1)]  $\Gamma_k $ are convex cones and $\Gamma_1\supset\Gamma_2\supset\cdots \supset\Gamma_n$;
\item [(\romannumeral2)]   $\sigma_{k-1}(\lambda|i)>0$ for $\lambda\in\Gamma_k,$ and $1\leq i\leq n$;
\item [(\romannumeral3)] If $\lambda\in\Gamma_k$  with $\lambda_1\geq \cdots\geq \lambda_k\geq \cdots \geq \lambda_n$, then
~$$\sigma_{k-1}(\lambda|k)\geq C(n,k)\sigma_{k-1}(\lambda).$$
\item [(\romannumeral4)] Newton-MacLaurin inequality: If $\lambda\in\Gamma_k,~n\geq k> l \geq 0$, $ n\geq r > s \geq 0$, $k\geq r$, $l\geq s$, then
\begin{equation}\label{lem21}
\Bigg[\frac{{\sigma _k (\lambda )}/{C_n^k }}{{\sigma _l (\lambda
)}/{C_n^l }}\Bigg]^{\frac{1}{k-l}} \le \Bigg[\frac{{\sigma _r
(\lambda )}/{C_n^r }}{{\sigma _s (\lambda )}/{C_n^s
}}\Bigg]^{\frac{1}{r-s}}. \notag
\end{equation}
\item [(\romannumeral5)]  If $\lambda\in\Gamma_k,~0\leq l<k\leq n,$~then
$$\sum\limits_{i=1}^n \frac{\partial \bigg[\frac{\sigma_k(\lambda)}{\sigma_l(\lambda)}\bigg]^{\frac{1}{k-l}}}{\partial \lambda_i}\geq
\bigg[\frac{C^k_n}{C^l_n}\bigg]^{\frac{1}{k-l}}.$$

\item [(\romannumeral6)]  If $\lambda\in\Gamma_k,~0\leq l<k\leq n$,~then $\bigg[\frac{\sigma_k(\lambda)}{\sigma_l(\lambda)}\bigg]^{\frac{1}{k-l}}$ are concave functions in $\Gamma_k$.

\end{enumerate}
\end{proposition}

 For $\lambda=(\lambda_1,\lambda_2,\cdots,\lambda_n)\in \mathbb{R}^n$ and $1\leq p\leq n-1$, recall the notation $\Lambda(\lambda)=(\Lambda_{I_1},\cdots,\Lambda_{I_N})$ that
$$\Lambda_{I_s}=\sum_{i_j\in I_s} \lambda_i=\lambda_{i_1}+\cdots+\lambda_{i_p}, \quad I_s=(i_1, \cdots, i_p)\in \mathfrak{J}.$$
Then we have the following basic properties. The proof can refer to Lemma 2.3-2.5 of \cite{ZJ}.

\begin{proposition}\label{P4}
Let $\lambda=(\lambda_1,\lambda_2,\cdots,\lambda_n)\in \mathbb{R}^n$ with $\lambda_1\geq \lambda_2\geq\cdots\geq \lambda_n$ and $\Lambda(\lambda)\in \tilde{\Gamma}_k$. If
$\Lambda(\lambda)=(\Lambda_{I_1},\cdots,\Lambda_{I_N})$ with $\Lambda_{I_1}\geq \cdots\geq \Lambda_{I_N}$, then
\begin{enumerate}
\item[(\romannumeral1)]  $\Lambda_{I_1}=\lambda_1+\cdots+\lambda_p$ and $\Lambda_{I_N}=\lambda_{n-p+1}+\cdots+\lambda_n$.
\item [(\romannumeral2)]   We have
\begin{align}\label{2.3}
&\frac{\partial \bigg[ \frac{\sigma_k(\Lambda(\lambda))}{\sigma_l(\Lambda(\lambda))}\bigg]}{\partial \Lambda_{I_1}}\leq \cdots\leq \frac{\partial \bigg[ \frac{\sigma_k(\Lambda(\lambda))}{\sigma_l(\Lambda(\lambda))}\bigg]}{\partial \Lambda_{I_N}},\quad 0\leq l<k\leq N,
\end{align}
and
\begin{align}\label{2.4}
&\frac{\partial \bigg[ \frac{\sigma_k(\Lambda(\lambda))}{\sigma_l(\Lambda(\lambda))}\bigg]}{\partial \lambda_{1}}\leq \cdots\leq \frac{\partial \bigg[ \frac{\sigma_k(\Lambda(\lambda))}{\sigma_l(\Lambda(\lambda))}\bigg]}{\partial \lambda_{n}},\quad 0\leq l<k\leq N.
\end{align}
\item [(\romannumeral3)] If $0\leq l<k\leq C^{p-1}_{n-1}$, then there exists a  positive constant $C_p$ depending on $n, k, l, p$ such that
\begin{equation}\label{2.5}
\frac{\partial\bigg[\frac{\sigma_k(\Lambda(\lambda))}{\sigma_l(\Lambda(\lambda))}\bigg]^{\frac{1}{k-l}}}{\partial \lambda_i}\geq C_p
\sum\limits_{i=1}^n\frac{\partial\bigg[\frac{\sigma_k(\Lambda(\lambda))}{\sigma_l(\Lambda(\lambda))}\bigg]^{\frac{1}{k-l}}}{\partial \lambda_i},~~~~1\leq i\leq n;
\end{equation}
\item [(\romannumeral4)] If $0\leq l<k\leq N$, then
$\bigg[\frac{\sigma_k(\Lambda(\lambda))}{\sigma_l(\Lambda(\lambda))}\bigg]^{\frac{1}{k-l}}$ are concave with respect to $\lambda$ and
\begin{equation}\label{2.9}
\sum\limits_{i=1}^n\frac{\partial\bigg[\frac{\sigma_k(\Lambda(\lambda))}{\sigma_l(\Lambda(\lambda))}\bigg]^{\frac{1}{k-l}}}{\partial \lambda_i}\geq \tilde{C_p}:=p\bigg(\frac{C_N^k}{C_N^l}\bigg)^{\frac{1}{k-l}}>0.
\end{equation}
\end{enumerate}
\end{proposition}

For convenience, we make the following notations
$$
F(D^2 u):=\bigg[\frac{\sigma_k(\Lambda(D^2 u))}{\sigma_l(\Lambda(D^2 u))}\bigg]^{\frac{1}{k-l}}, \quad F^{ij}=\frac{\partial F}{\partial u_{ij}}, \quad \quad~F^{ij,kl}=\frac{\partial^2 F}{\partial u_{ij} u_{kl}}, \quad \tilde{\psi}=\psi^{\frac{1}{k-l}}.$$
Equation \eqref{(1.1)} can be written by
$$F(D^2 u)= \tilde{\psi}(x, u, Du).$$
Then we have

\begin{proposition}\label{P7}
For any symmetric matrix $(\eta_{ij})$, we have
$$F^{ij,kl}\eta_{ij}\eta_{kl}=F^{ii,jj}\eta_{ii}\eta_{jj}+\sum_{i\ne j} \frac{F^{ii}-F^{jj}}{u_{ii}-u_{jj}}\eta^2_{ij}.$$
Moreover, the second term at the right-hand of the equation is non-positive.
\end{proposition}
\begin{proof}
The proof can be found in \cite{CG}.
\end{proof}

\section{$C^0$ estimates}
In this section, we consider the $C^0$ estimates for the $(\Lambda, k)$-convex solution of equation \eqref{(1.4)}.

\begin{theorem}\label{C0}
Let $u\in C^2(\Omega)\cap C^1(\overline\Omega)$ be a $(\Lambda, k)$-convex solution of Neumann problem~\eqref{(1.4)}. Suppose that there exists positive constants $c_0, \alpha_0$ such that $-\varphi_u\geq c_0>0$ and  $\tilde{\psi}^{-1}_u\leq -\alpha_0<0$,~then
$$|u|_{C^0}\leq C,$$
where~$C$~is a positive constant depending on $n, k, l, p, c_0, \alpha_0, |\varphi|_{C^0}$ and $\Omega$.
\end{theorem}

\begin{proof}
On the one hand, note that $\Delta u>0$ since $\Lambda(D^2 u)\in \tilde{\Gamma}_k\subset \tilde{\Gamma}_1$, hence $u$ reaches maximum value at $x_0 \in \partial \Omega $. Assuming $u(x_0)>0$, otherwise $u $ would have an upper bound, then
\begin{eqnarray*}
\begin{split}
0\leq u_\nu&=\varphi(x_0,u(x_0))-\varphi(x_0,0)+\varphi(x_0,0)\\
&=u(x_0)\varphi_u(x_0,tu(x_0))+\varphi(x_0,0)\\
&\leq -c_0u(x_0)+\varphi(x_0,0).
\end{split}
\end{eqnarray*}
where $t\in (0,1)$, it implies
$$u\leq C.$$

On the other hand, we need to obtain the lower bound for $u$. Consider the auxiliary function
$$H(x)=u^2(x)-s(x)u(x),$$
where the smooth function $s$ satisfies the following properties in $\Omega $
$$s>0, \quad D^2 s\geq \lambda_0I>0.$$
Here $\lambda_0$ is a positive constant.
Let $H(x)$ achieves the maximum value at point $x_1\in \bar{\Omega}$. We assuming $u(x_1)<-\frac{{|s|_{C^1}}^2}{\lambda_0}$, otherwise $u$ would have an lower bound. We discuss the situation separately.

\textbf{Case 1: }  $x_1\in \Omega $. At point $x_1$, we have
$$H_i=2uu_i-s_iu-su_i=0,$$
and
\begin{eqnarray*}
\begin{split}
0\geq H_{ij}=(2u-s)u_{ij}+2s_is_j\bigg(\frac{u}{2u-s}- \frac{1}{2}\bigg)^2-\frac{1}{2}s_is_j-s_{ij}u.
\end{split}
\end{eqnarray*}
Then
\begin{eqnarray*}
\begin{split}
0\geq F^{ij}H_{ij}\geq&(2u-s)\tilde{\psi}-\frac{1}{2}\sum F^{ii} |Ds|^2+(-u)F^{ij}s_{ij}\\
\geq&(2u-s)\tilde{\psi}+\frac{\lambda_0}{2}(-u)\sum_i F^{ii}.
\end{split}
\end{eqnarray*}
Due to $\tilde{\psi}>0$ and ~$\sum\limits_i F^{ii}\geq \tilde{C_p}$, we get
\begin{eqnarray*}
\begin{split}
0 \geq &-2+\frac{s}{u}+\frac{\lambda_0}{2}\tilde{C_p} \tilde{\psi}^{-1}(x_1,u(x_1),Du(x_1))\\
\geq &-2+\frac{s}{u}+\frac{\lambda_0}{2}\tilde{C_p}\bigg[\bigg(\tilde{\psi}^{-1}(x_1,tu(x_1),Du(x_1))\bigg)_uu(x_1)+\tilde{\psi}^{-1}(x_1,0,Du(x_1))\bigg]\\
\geq &-2+\frac{-s(x_1)}{-u(x_1)}+\frac{\lambda_0}{2}\tilde{C_p}\alpha_0(-u(x_1)),
\end{split}
\end{eqnarray*}
where $t\in (0,1)$. It implies that $-u(x_1)\leq C$, and then $H(x)\leq H(x_1)\leq C$. Thus we obtain
$$u(x)\geq -C,\quad\forall x \in\bar{\Omega}.$$

\textbf{Case 2: } $x_1\in\partial \Omega$. Then,
\begin{eqnarray*}
\begin{split}
0\leq &D_\nu H(x_1)=(2u-s)\varphi-uD_\nu s\\
=&2\varphi_u(x_1,tu(x_1)) u^2+u(2\varphi(x_1,0)-D_\nu s-s\varphi_u(x_1,tu(x_1)))-s\varphi(x_1,0)\\
\leq& -2c_0u^2-C_1u+C_1
\end{split}
\end{eqnarray*}
 where $t\in (0,1)$. Similarly we get
$|u(x_1)|\leq C$ and hence
$$u(x)\geq -C,\quad\forall x\in\bar{\Omega}.$$
Then the proof is completed.
\end{proof}
\
\section{global gradient estimates}
In this section, we consider the global gradient  estimates for the $(\Lambda, k)$-convex solution of equation \eqref{(1.4)}. We always assume that the conditions in Theorem \ref{T1} hold.

\subsection{Internal gradient estimates}
In this subsection, we establish the internal gradient estimates, i.e. Theorem \ref{C11}, under  \textbf{Condition \ref{growth}}.

\begin{proof}[\textbf{Proof of Theorem \ref{C11}}]
Set auxiliary function
\begin{eqnarray*}
\phi=\log |D u|+\log \eta (u)+\log \rho,
\end{eqnarray*}
where
$$\rho=r^2-|x|^2,  \quad \eta(u)=(M-u)^{-\frac{1}{3}}, \quad M=3\underset{B_r(0)}{osc} u+\underset{B_r(0)}{sup} u.$$
Assuming that $\phi $ takes its maximum value at point  $x_0 \in B_r(0) $. By rotating the coordinate system $\{x_1,\cdots,x_n\}$ so that $D^2 u(x_0)$ is diagonal and $u_{11}(x_0)=\lambda_1(x_0)\geq \cdots\geq u_{nn}(x_0)= \lambda_n(x_0)$. Let $T=\underset{B_r(0)}{osc} u$. All calculations  are  at point $x_0 $. We have
\begin{eqnarray}\label{324}
0=\phi_i=\frac{\sum_k u_k u_{ki}}{|Du|^2}+\frac{\eta_i}{\eta}+\frac{\rho_i}{\rho}.
\end{eqnarray}
Calculate second-order derivative
\begin{eqnarray}\label{325}
\begin{split}
0\geq \phi_{ij}=&\frac{\sum_k u_k u_{ijk}}{|Du|^2}+\frac{\sum_k u_{kj}u_{ki}}{|Du|^2}-\frac{2\sum_l u_lu_{lj}\sum_k u_ku_{ki}}{|Du|^4}\\
&+\frac{\eta_{ij}}{\eta}-\frac{\eta_i\eta_j}{\eta^2}+\frac{\rho_{ij}}{\rho}-\frac{\rho_i\rho_j}{\rho^2}.
\end{split}
\end{eqnarray}
Combining with \eqref{324}  and \eqref{325},
\begin{align*}
\begin{split}
0&\geq F^{ij}\phi_{ij}\\
&\geq\frac{1}{|Du|^2} \bigg( \sum_k\left(\tilde{\psi}_{x_k}u_k+\tilde{\psi}_u u_k^2\right)-\sum_i \tilde{\psi}_{p_i}(\frac{\eta_i}{\eta}+\frac{\rho_i}{\rho})|Du|^2\bigg)+F^{ij}\bigg(\frac{\eta_{ij}}{\eta}-3\frac{\eta_i\eta_j}{\eta^2}\bigg)\\
&\quad+F^{ij}\bigg(\frac{\rho_{ij}}{\rho}-3\frac{\rho_i\rho_j}{\rho^2}\bigg)-2F^{ij}\bigg(\frac{\eta_i\rho_j}{\eta\rho}+\frac{\eta_j\rho_i}{\eta\rho}\bigg)\\
&\geq \frac{1}{|Du|^2}\bigg(-|D_x\tilde{\psi}||Du|-|D_u\tilde{\psi}||Du|^2-|D_p\tilde{\psi}||Du|^3\frac{\eta'}{\eta}-|D_p\tilde{\psi}||Du|^2\frac{|D{\rho}|}{\rho}\bigg)\\
&\quad +F^{ii}u_i^2\bigg(\frac{\eta''}{\eta}-3\frac{\eta'^2}{\eta^2}\bigg)+\tilde{\psi}\frac{\eta'}{\eta}+ \sum_i F^{ii}\bigg(\frac{-2}{\rho}-3\frac{|D{\rho}|^2}{\rho^2}\bigg)
-4\sum_i F^{ii}\frac{\eta'|D{\rho}||D u|}{\eta \rho}\\
&\geq-\frac{|D_x \tilde{\psi}|}{|Du|}-|D_u\tilde{\psi}|-\frac{1}{9T}|D_p\tilde{\psi}||Du|-|D_p\tilde{\psi}|\frac{2r}{\rho}+C_p\sum_i F^{ii}|Du|^2\frac{1}{144T^2}\\
&\quad-\sum_i F^{ii}\bigg(\frac{2}{\rho}+\frac{12r^2}{\rho^2}+\frac{8r|Du|}{9T\rho}\bigg)\\
&\geq-C_1 \bigg({|Du|}^{1+\gamma}+\frac{{|Du|}^{1+\gamma}}{9T}+\frac{2r|Du|^{\gamma}}{\rho}\bigg)
+C_p\sum_i F^{ii}|Du|^2\frac{1}{144T^2}\\
&\quad -\sum_i F^{ii}\bigg(\frac{2}{\rho}+\frac{12r^2}{\rho^2}+\frac{8r|Du|}{9T\rho}\bigg)
\end{split}
\end{align*}
if $|D u|>M_1$, in the last inequality, we use the \textbf{Condition \ref{growth}}.
Then we get
$$({\sqrt{C_p}}|Du|\rho-\frac{64}{\sqrt{C_p}}Tr)^2\leq (\frac{{64}^2}{C_p}+144\cdot14)T^2r^2+CT^2\rho^2\bigg({|Du|}^{1+\gamma}+\frac{{|Du|}^{1+\gamma}}{T}+\frac{r|Du|^{\gamma}}{\rho}\bigg),$$
where $C$ is a positive constant.
It is easy to show that
$$|Du|\rho(x_0)\leq C(Tr+T^\frac{2}{1-\gamma}r^2+T^\frac{1}{1-\gamma}r^2+T^{\frac{2}{2-\gamma}} r^{\frac{3-2\gamma}{2-\gamma}}).$$
So
\begin{align*}
\begin{split}
|Du|(0)&\leq \frac{\eta(u)(x_0)}{\eta(u)(0)}\frac{1}{\rho(0)}\rho(x_0)|Du|(x_0)\\
&\leq{( \frac{4}{3})}^{-\frac{1}{3}}\frac{1}{r^2}|Du|(x_0)\rho(x_0)\\
&\leq{( \frac{4}{3})}^{-\frac{1}{3}}C(\frac{T}{r}+T^{\frac{2}{1-\gamma}}+T^{\frac{1}{1-\gamma}}+T^{\frac{2}{2-\gamma}} r^{-\frac{1}{2-\gamma}})\\
&\leq {C(\frac{T}{r}+T^{\frac{2}{1-\gamma}}+T^{\frac{1}{1-\gamma}})}.
\end{split}
\end{align*}
where $C$ is a positive constant depending on $k, l, n, p, M_1$ and $C_1$.

\begin{remark}
If $\tilde{\psi}$ is a constant, then
$$|Du|\rho(x_0)\leq CTr,$$
thus we can obtain
$$|Du|(0)\leq C\frac{T}{r},$$
without the  \textbf{Condition \ref{growth}}.
\end{remark}


\end{proof}

\subsection{Boundary gradient estimates}
\ \ \\

In this section, we give a very small positive constant $\mu$ to define the boundary band of $\partial \Omega$ as
$$\Omega_\mu:=\{x\in\Omega,~0<d(x,~\partial \Omega)<\mu\}.$$
\begin{theorem}\label{C12}
Assuming that $u \in C^3(\Omega)$ is a $(\Lambda, k)$-convex solution for equation \eqref{(1.4)}, $\tilde{\psi}>0$ has the growth \textbf{Condition \ref{growth}}. Then there exists a constant $\mu$ depending on $|\varphi|_{C^2}$ such that
$$\underset{\bar\Omega_\mu}{\sup}|D u|\leq C,$$
where $C$ is a positive constant depending on $|u|_{C^0}, |\varphi|_{C^2}, C_1, M_1, \mu, n, k, l, p$ and $\Omega$.
\end{theorem}

\begin{proof}
We define the following auxiliary function
$$G(x)=\log |D w|^2+h(u)+Ad(x).$$
where $w(x,u)=u(x)+\varphi(x,u)d(x)$, $h(u)=\delta(u+C_0)^2$ with $C_0\geq |u|_{C^0}+1$ and $\delta<\frac{1}{32{C_0}^2}$, $A$ needs to be confirmed later.

Assuming that $G$ reaches its maximum value at point $x_0$, we only need to show that
$$|D u|(x_0) \leq C .$$
Without loss of generality, we assume
$$\mu <\frac{1}{4|\varphi|_{C^1}}\quad\mbox{and} \quad |D u|(x_0) \geq 8|d|_{C^2}|\varphi|_{C^1}+2.$$
then
\begin{eqnarray}\label{1.3.1}
1<\frac{1}{2}|D u|(x_0) \leq |D w|(x_0) \leq \frac{3}{2}|D u|(x_0) .
\end{eqnarray}
We prove the theorem  in three cases.

\textbf{Case 1:} If $x_0\in \partial \Omega$, select the local orthogonal coordinate system at $x_0 $ so that $e_n=\nu$ and $e_1,\cdots, e_{n-1}$ are tangent to boundary  $\partial \Omega$. Then at point $x_0$
\begin{eqnarray}\label{3.1}
0\leq G_\nu=\frac{D_\nu |D w|^2}{|D w|^2}+h'D_\nu u+AD_\nu d,
\end{eqnarray}
Note that $D_\nu d=-1$ and $d=0$ on $\partial \Omega$,
$$D_\nu w=D_\nu u+\varphi D_\nu d=0$$
So  the inequality \eqref{3.1} can be rewritten as
$$0\leq \frac{2D_\alpha w D_{\nu\alpha} w}{|D w|^2}+h'D_\nu u-A,$$
where $\alpha$ represents the sum of 1 to $n-1 $. According to the boundary conditions, we have
\begin{eqnarray*}
\begin{split}
D_{\nu\alpha} w&=D_{\alpha\nu}u+dD_{\alpha\nu} \varphi+\varphi D_{\alpha\nu} d+D_\alpha \varphi D_\nu d+D_\alpha dD_\nu \varphi\\
&=-D_{D_{\alpha}\nu}u+\varphi D_{\alpha\nu} d +D_\alpha dD_\nu \varphi.
\end{split}
\end{eqnarray*}
By \eqref{1.3.1}, it is easy to show
$$\left|D_{\nu\alpha} w \right|\leq C|D u|,$$
where $C$ is a positive constant depending on $|\varphi|_{C^1}$ and $\Omega$.
Therefore, we have
$$0\leq C+h'|\varphi|_{C^0}-A,$$
 we can get a contradiction by choosing $A>(2C + \frac {|\varphi|_{C^0}}{4C_0})$. Thus, $| D u |(x_0) \leq C $.

\textbf{Case 2:} If $x_0\in \Omega_\mu$, calculate the partial derivative of $G $ twice at $x_0 $
\begin{eqnarray}\label{eq322}
0=G_j=\frac{2\sum_k w_kw_{jk}}{|D w|^2}+h'u_j+Ad_j,
\end{eqnarray}
and
\begin{eqnarray}\label{eq323}
\begin{split}
0\geq G_{ij}=&\frac{2\sum_k w_kw_{ijk}}{|D w|^2}+\frac{2\sum_k w_{ik} w_{jk}}{|D w|^2}-\frac{4\sum_k w_kw_{jk}\sum_l w_lw_{li}}{|D w|^4}\\
&+h''u_i u_j +h'u_{ij}+Ad_{ij} .
\end{split}
\end{eqnarray}
Combining with \eqref{eq322} and \eqref{eq323}, we have
\begin{eqnarray}\label{eq324}
&&0\geq F^{ij} G_{ij}\\
\nonumber&& =\frac{2}{|D w|^2} F^{ij}\sum_k w_k w_{ijk} +\frac{2}{|D w|^2}F^{ij}\sum_k w_{ik} w_{jk}-F^{ij}(h'u_j+Ad_j)(h'u_i+Ad_i)\\
\nonumber&&+F^{ij}h''u_i u_j +F^{ij}h'u_{ij} +F^{ij}Ad_{ij} \\
\nonumber&&\geq\frac{2}{|D w|^2} F^{ij}\sum_k w_k w_{ijk} +(h''-2h'^2)F^{ij}u_i u_j-2A^2F^{ij}d_i d_j \\
\nonumber&&+h'F^{ij}u_{ij}+AF^{ij}d_{ij}.
\end{eqnarray}
For convenience, we denote $\Phi(x,u)=\varphi(x,u)d(x)$, then $w=u-\Phi$. Directly calculate the partial derivative of $w $ to obtain
\begin{eqnarray*}
w_k =u_k -\Phi_u u_k -\Phi_{x_k}=(1-\Phi_u)u_k -\Phi_{x_k},
\end{eqnarray*}
Differentiating $w_k$ again, we have
$$w_{jk} =(1-\Phi_u)u_{kj} -(\Phi_{ux_j}+\Phi_{uu}u_j )u_k  -\Phi_{x_kx_j}-\Phi_{x_k u}u_j .$$
Hence
\begin{eqnarray*}
\begin{split}
&(1-\Phi_u)\sum_k w_k u_{jk}\\
&=\sum_k w_kw_{kj}+(\Phi_{ux_j}+\Phi_{uu}u_j )\sum_k w_ku_k+ \sum_k \Phi_{x_kx_j}w_k+\sum_k\Phi_{x_k u}u_jw_k\\
&=-\frac{1}{2}|D w|^2(h'u_j+Ad_j)+\sum_k(\Phi_{uu}u_jw_ku_k+\Phi_{ux_j}w_ku_k+\Phi_{x_kx_j}w_k+\Phi_{x_k u}u_jw_k),
\end{split}
\end{eqnarray*}
Note that $|\Phi_u|+|\Phi_{x_k}|+|\Phi_{ux_k}|+|\Phi_{uu}|+|\Phi_{x_kx_j}|\leq C$.
Combining with \eqref{1.3.1} and the fact $h'(u) \leq 4 \delta C_0$, assuming $|Du|(x_0) > \max\{A, 8|d|_{C^2}|\varphi|_{C^1}+2, M_1\} $,
we obtain
\begin{eqnarray*}
\begin{split}
\left|\sum_k w_k u_{kj}\right|&\leq \frac{C}{1-\mu|\varphi|_{C^1}}(4 \delta C_0|D u|^3+A|Du|^2+ |Du|^3+|Du|^2+|Du|)\\
&\leq C|Du|^3
\end{split}
\end{eqnarray*}
From the growth \textbf{Condition \ref{growth}}, we obtain
$$ \left|F^{ij}u_{ijk} w_k \right| =\left| w_k(\tilde{\psi}_{x_k}+\tilde{\psi}_u u_k)  +\sum_l \tilde{\psi}_{p_l}u_{kl} w_k\right| \leq C|D u|^{3+\gamma}.$$
Note that
\begin{align*}
\begin{split}
w_{ijk} =&(1-\Phi_u) u_{ijk}-(\Phi_{ux_i}+\Phi_{uu} u_i) u_{jk}-(\Phi_{ux_j}+\Phi_{uu} u_j) u_{ik}\\
&-(\Phi_{ux_jx_i}+\Phi_{ux_ju} u_i+\Phi_{uux_i} u_j+\Phi_{uuu} u_i u_j+\Phi_{uu} u_{ij}) u_k\\
&-\Phi_{x_ku} u_{ij}-(\Phi_{x_kux_i}+\Phi_{x_kuu} u_i) u_j-\Phi_{x_kx_jx_i}-\Phi_{x_kx_ju} u_i.
\end{split}
\end{align*}
Then
\begin{align*}
\begin{split}
&F^{ij}\sum_k w_kw_{ijk}\\
=&(1-\Phi_u)F^{ij}\sum_k w_k u_{ijk} -2F^{ij}\sum_k w_ku_{jk}(\Phi_{ux_i}+\Phi_{uu}u_i )-F^{ij}u_{ij}\sum_k w_k (\Phi_{uu}u_k+\Phi_{x_ku})\\
&- F^{ij}\sum_k w_ku_k(\Phi_{ux_jx_i}+\Phi_{ux_ju} u_i+\Phi_{uux_i} u_j+\Phi_{uuu} u_i u_j )\\
&-F^{ij} \sum_k w_k u_j(\Phi_{x_kux_i}+ \Phi_{x_kuu} u_i)-F^{ij}\sum_k w_k\Phi_{x_kx_jx_i}-F^{ij} \sum_k w_ku_i\Phi_{x_kx_ju}\\
\geq &-C|D u|^{3+\gamma}- C \tilde{\psi} (\mu |Du|^2+|Du|)- C \mu \sum\limits_i F^{ii} |D u|^4- C\sum\limits_i F^{ii} |D u|^3\\
\end{split}
\end{align*}
According to the definition $h(u)=\delta(u+C_0)^2$, taking $|u|_{C^0}+1\leq C_0$. Then
$$h'(u)=2\delta(u+C_0)\geq 2\delta, \quad h''(u)=2\delta.$$
Due to $\delta<\frac{1}{32C_0^2}$, we have
$$h''-2h'^2=2\delta-8\delta^2(u+C_0)^2>\delta>0.$$
Without loss of generality, we assume that $\{u_ {ij} (x_0)\} $ is diagonal with
$$ u_{11}(x_0)\geq u_{22}(x_0)\geq\cdots\geq u_{nn}(x_0)$$
According to proposition \ref{P4} (\romannumeral3), we obtain
$$F^{ij}u_i u_j =F^{ii}(u_i)^2\geq F^{11}\sum_i u_i^2  \geq C_p|D u|^2\sum\limits_i F^{ii}.$$
where $C_p$ is a positive constant defined in  Proposition \ref{P4}(\romannumeral3).\\

Choosing
$$\mu< \min \{\frac{1}{4|\varphi|_{C^1}}, \frac{C_p \delta}{4C}, \frac{ \delta}{C}\}, \quad |Du|(x_0) \geq \max\left\{ \frac{C}{\delta},~\left({\frac{4C}{C_p \delta}} +1\right)A,~ 8|d|_{C^2}|\varphi|_{C^1}+2, M_1 \right\}.$$
Then  \eqref{eq324} implies that
\begin{eqnarray*}
\begin{split}
0\geq &\bigg(- C|D u|^{1+\gamma} -C\tilde{\psi} (\mu+\frac{1}{|Du|})-C \mu \sum_i F^{ii}|D u|^2- C\sum_i F^{ii}|D u|\bigg) \\
&+C_p\delta\sum_i F^{ii} |Du|^2-CA^2\sum_iF^{ii} +2\delta\tilde{\psi} \\
\geq &\sum_i F^{ii}\bigg(C_p\delta|Du|^2-C\mu |D u|^2-C|D u|-CA^2\bigg)+\tilde{\psi}(2\delta-C\mu-\frac{C}{|Du|})-C|D u|^{1+\gamma}\\
\geq&\frac{\tilde{C}_pC_p \delta}{4} |Du|^2-C|D u|^{1+\gamma}
\end{split}
\end{eqnarray*}
where the constant $\tilde{C}_p$ is defined in Proposition \ref{P4}(\romannumeral4).
 Thus, $| D u |(x_0) \leq C $.

\textbf{Case 3:} If $x_0\in\partial \Omega_\mu\backslash\partial \Omega$, from theorem \ref{C11}, we can know that $|D u|(x_0)$ is bounded.

Combining with Case 1, Case 2 and Case 3, the boundary gradient estimates is completed.
\end{proof}

In summary, we can obtain
$$\underset{\bar{\Omega}}{\sup}|D u|\leq C.$$
where $C$ is a positive constant depending on $|u|_{C^0}, |\varphi|_{C^2}, C_1, M_1,  \mu, n, k, l, p$ and $\Omega $.

\
\section{global estimates for second-order derivatives}
In this section, we consider the global second-order  derivatives estimates for the $(\Lambda, k)$-convex solution of equation \eqref{(1.4)}. More precisely, we reduce the global  second-order  derivatives estimates into the boundary  double normal estimates, and establish double normal estimates on the boundary.

\subsection{Reduce global  second-order derivatives to double normal  second-order derivatives on the boundary}

\begin{theorem}\label{C21}
Let $u\in C^4(\Omega)$~be a ~$(\Lambda,k)$-convex solution  of equation \eqref{(1.4)},~then
$$
\sup\limits_{(x,\zeta)\in \Omega\times{\mathbb{S}^{n-1}}}D_{\zeta \zeta}u(x)\leq C(1+\max\limits_{\partial\Omega}|D_{\nu\nu}u|),
$$
where~$C$~is a positive constant depending on $n, k, l, p, |u|_{C^1}, |\tilde{\psi}|_{C^2}, |\varphi|_{C^3}$  and $ \Omega$.
\end{theorem}

\begin{proof}
Define $h(R)=e^{-AR}$ with $R(x)\in C^2(\bar{\Omega})$, $R|_{\partial \Omega}=0$ and  $D_\nu R=1$ on $\partial\Omega$. $A=1+2\max_{\partial \Omega} \{|\Pi_{ij}|\}+|\varphi|_{C^2}$, $|\Pi_{ij}|$ is the second fundamental form of the boundary. We consider such an auxiliary function
$$ \Phi(x,\zeta)=h(R)(D_{\zeta\zeta}u-v( x,\zeta))+B|Du|^2,$$
where $v(x,\zeta)=2\langle \zeta,\nu\rangle\langle \zeta', D_x\varphi+D_u\varphi D u-D_luD \nu^l\rangle\equiv a^lD_l u+b$, $\nu\in \mathbb{S}^{n-1}$ is a $C^3(\overline{\Omega})$ extension of the outer unit normal vector field on $\partial \Omega$,  $\zeta'=\zeta-\langle \zeta,\nu\rangle\nu$, $a^l=2\langle \zeta,~\nu \rangle(D_u\varphi(\zeta')^l-\langle\zeta',D\nu^l \rangle)$
, $b=2\langle \zeta,\nu\rangle\langle \zeta',D_x\varphi\rangle$  and $B$ is a positive constant that needs to be confirmed.

Suppose $\Phi(x,\zeta)$ achieves the maximum at $(x_0,\zeta_0)$, and hence
 $\Phi(x,\zeta_0)$ achieves the maximum at point $x_0$.

 We now separate into two cases.

\textbf{Case 1: }$x_0\in \Omega.$
Rotating the coordinate system at point $x_0$ so that $D ^ 2u $ is diagonal and $u_{11}(x_0)\geq u_{22}(x_0)\geq \cdots\geq u_{nn}(x_0)$. Calculate the partial derivative of $\Phi $  at   $x_0$,
$$0=\Phi_i(x_0)=h'R_i\bigg(D_{\zeta_0\zeta_0}u-v\bigg)+h\bigg(D_{\zeta_0\zeta_0}u-v\bigg)_i+2B\sum_ku_ku_{ki}.$$
Then
\begin{eqnarray}\label{eq421}
\begin{split}
\bigg(D_{\zeta_0\zeta_0}u-v\bigg)_i=-\frac{h'R_i\bigg(D_{\zeta_0\zeta_0}u-v\bigg)+2B\sum\limits_ku_ku_{ki}}{h}.
\end{split}
\end{eqnarray}
Calculate the second-order partial derivative of $\Phi$ at $x_0$, we have
\begin{eqnarray*}
\begin{split}
0\geq &F^{ij}\Phi_{ij}\\
=&(h''R_iR_j+h'R_{ij})F^{ij}\bigg(D_{\zeta_0\zeta_0}u-v\bigg)+hF^{ij}\bigg(D_{\zeta_0\zeta_0}u-v\bigg)_{ij}+2BF^{ij}\sum_k( u_{ki}u_{kj}+u_ku_{kij})\\
&+h'R_iF^{ij}\bigg(D_{\zeta_0\zeta_0}u-v\bigg)_j+h'R_jF^{ij}\bigg(D_{\zeta_0\zeta_0}u-v\bigg)_i\\
=&hF^{ij}u_{ij\zeta_0\zeta_0}-hF^{ij}v_{ij}+\bigg(h''-2\frac{h'^2}{h}\bigg)F^{ij}R_iR_j(D_{\zeta_0\zeta_0}u-v)+h'F^{ij}R_{ij}(D_{\zeta_0\zeta_0}u-v)\\
&-4BF^{ij}\frac{h'}{h}\sum_ku_ku_{ki}R_j
+2BF^{ij}\sum_k u_{ki}u_{kj}+2B\sum_k u_k \tilde{\psi}_k,
\end{split}
\end{eqnarray*}

According to proposition  \ref{P7} and \eqref{eq421}, we have
\begin{eqnarray*}
\begin{split}
F^{ij}u_{ij\zeta_0\zeta_0}\geq & \tilde{\psi}_{\zeta_0\zeta_0}\\
=&\tilde{\psi}_{x_ix_j}({\zeta_0})^i ({\zeta_0})^j +\tilde{\psi}_{x_iu} u_{\zeta_0}({\zeta_0})^i+\sum_l\tilde{\psi}_{x_ip_l}u_{l\zeta_0}({\zeta_0})^i +\tilde{\psi}_u u_{\zeta_0\zeta_0}+\tilde{\psi}_{u x_i}u_{\zeta_0}({\zeta_0})^i\\
&+\tilde{\psi}_{uu}u_{\zeta_0}^2
+\sum_l \tilde{\psi}_{up_l}u_{\zeta_0}u_{l\zeta_0}
+\sum_l \tilde{\psi}_{p_lx_i}u_{l\zeta_0}({\zeta_0})^i
+\sum_l\tilde{\psi}_{p_lu}u_{\zeta_0}u_{l\zeta_0}\\
&+\sum_l\sum_t\tilde{\psi}_{p_lp_t}u_{l\zeta_0}u_{t\zeta_0}
-\sum_l \tilde{\psi}_{p_l}\bigg(\frac{h'R_l\bigg(D_{\zeta_0\zeta_0}u-v\bigg)+2B\sum\limits_ku_ku_{kl}}{h}-v_l\bigg)\\
\geq&-C(1+u_{11}^2)-\sum_l \tilde{\psi}_{p_l}\bigg(\frac{h'R_l\bigg(D_{\zeta_0\zeta_0}u-v\bigg)+2B\sum\limits_ku_ku_{kl}}{h}\bigg)
\end{split}
\end{eqnarray*}
Hence
\begin{eqnarray}\label{eq423}
\begin{split}
&hF^{ij}u_{ij\zeta_0\zeta_0}+2B\sum_k u_k\tilde{\psi}_k\\
&\geq -Ch(1+u_{11}^2)-\sum_l \tilde{\psi}_{p_l}\bigg(h'R_l(D_{\zeta_0\zeta_0}u-v)+2B\sum\limits_ku_ku_{kl}\bigg)\\
&~~+2B\sum_k u_k(\tilde{\psi}_{x_k}+\tilde{\psi}_u u_k+\sum_l \tilde{\psi}_{p_l}u_{kl})\\
&\geq-Ch(1+u_{11}^2)-CAhu_{11}-CAh-CB.
\end{split}
\end{eqnarray}
where $C$ is a positive constant depending on $|\tilde{\psi}|_{C^2}, |u|_{C^1}, |\varphi|_{C^2}$ and $\Omega$.

According to the Cauchy-Schwarz inequality, it is easy to see that
$$ 4BF^{ij}\sum_k u_ku_{ki}\frac{h'}{h}R_j+2F^{ij}ha_i^l(D_l u)_j \leq  F^{ii} u_{ii}^2+CA^2B^2\sum_i F^{ii}+Ch^2\sum_i F^{ii},$$
then
\begin{eqnarray}\label{eq424}
&&-hF^{ij}v_{ij}-4B\frac{h'}{h}F^{ij}\sum_k u_ku_{ki}R_j\\
\nonumber&&=-hF^{ij}(a^l_{ij}D_l u+a^l(D_l u)_{ij}+b_{ij})-2hF^{ij}a_i^l(D_lu)_j-4B\frac{h'}{h}F^{ij}\sum_k u_ku_{ki}R_j\\
\nonumber&&\geq-Ch\tilde{\psi} -Ch\sum_i F^{ii}-Ch\tilde{\psi}_l-F^{ii} u_{ii}^2
-CA^2B^2\sum_i F^{ii}-Ch^2\sum_i F^{ii}\\
\nonumber&&\geq-C(h+h^2+A^2B^2)\sum_i F^{ii}-F^{ii} u_{ii}^2-Ch(1+u_{11})
\end{eqnarray}
In fact, $h'(R)=-Ae^{-AR}=-Ah$, $h''(R)=A^2e^{-AR}=A^2h$, then
$$h''-2\frac{h'^2}{h}=-A^2h,$$
Combine  \eqref{eq423}, \eqref{eq424} and $h\leq Ce^{CA}$ to get
\begin{eqnarray}\label{eq425}
\begin{split}
0\geq &F^{ij}\Phi_{ij}\\
\geq& -Ch(1+u_{11}^2)-CAhu_{11}-CAh-CB-C(h+h^2+A^2B^2)\sum_i F^{ii}\\
&-F^{ii} u_{ii}^2-Ch(1+u_{11})-CA^2h\sum_i F^{ii}(1+u_{11})+2BF^{ii} u_{ii}^2\\
\geq& (2B-1)F^{ii} u_{ii}^2-C\sum_i F^{ii}(e^{CA}+e^{2CA}+A^2B^2+A^2e^{CA}(1+u_{11}))\\
&-Ce^{CA}(1+u_{11}^2)-CAe^{CA}u_{11}-CAe^{CA}-CB\\
\geq&(2B-1)F^{ii} u_{ii}^2-C\sum_i F^{ii}(A^2e^{2CA}+A^2B^2+A^2e^{CA}u_{11})\\
&-CAe^{CA}(1+u_{11}^2)-CB
\end{split}
\end{eqnarray}
According to proposition \ref{P4} (\romannumeral3) and proposition \ref{P4} (\romannumeral4), \eqref{eq425} becomes

$$0\geq  C_pu_{11}^2(2B-1)-C(A^2e^{2CA}+A^2B^2+A^2e^{CA}u_{11})-(\frac{CAe^{CA}}{\tilde C_p}(1+u_{11}^2)+\frac{CB}{\tilde C_p}).$$
Choosing $B>3\left(\frac{CAe^{CA}}{\tilde C_p C_p}+\frac{CA^2e^{CA}}{C_p}\right)+1$ and assuming $u_{11}>1$, then we have
$$0\geq  C_pBu_{11}^2-C(A^2e^{2CA}+A^2B^2+\frac{CB}{\tilde C_p}).$$
Thus
$$u_{11} \leq C ,$$
where $C$  is a positive constant depending on $n, k, l, p, |\tilde{\psi}|_{C^2}, |u|_{C^1}, |\varphi|_{C^3}$ and $\Omega$.

Subsequently, we obtain
$$ |D_{\zeta_0\zeta_0}u|\leq C.$$

\textbf{Case 2: }$x_0\in \partial \Omega.$
This situation can be divided into two subcases.

\textbf{Subcase 2.1: }$\zeta_0$ is a non-tangential vector of $x_0\in\partial \Omega$. Denote by
$$\zeta_0=\langle \zeta_0,\theta\rangle\theta+\langle\zeta_0,\nu\rangle\nu,$$
where $\theta$ is a tangential vector, i.e. $\langle\theta, \nu\rangle=0$, then
$$D_{\zeta_0} u(x_0)=\langle \zeta_0,\theta \rangle D_\theta u(x_0)+\langle \zeta_0,\nu \rangle D_\nu u(x_0),$$
and
\begin{eqnarray*}
\begin{split}
D_{\zeta_0\zeta_0}u(x_0)&=\langle\zeta_0,\theta  \rangle^2D_{\theta\theta}u(x_0)+\langle \zeta_0,\nu \rangle^2D_{\nu\nu}u(x_0)+2\langle\zeta_0,\theta \rangle\langle\zeta_0,\nu \rangle D_{\theta\nu}u(x_0)\\
&=\langle\zeta_0,\theta  \rangle^2D_{\theta\theta}u(x_0)+\langle \zeta_0,\nu \rangle^2D_{\nu\nu}u(x_0)+2\langle\zeta_0,\nu \rangle[\zeta_0-\langle\zeta_0,\nu\rangle\nu][D\varphi-D_luD\nu^l]\\
&=\langle\zeta_0,\theta  \rangle^2D_{\theta\theta}u(x_0)+\langle \zeta_0,\nu \rangle^2D_{\nu\nu}u(x_0)+v(x_0,\zeta_0)
\end{split}
\end{eqnarray*}
By the fact $\Phi(x_0,\theta)=h(R)D_{\theta\theta}u+B|Du|^2$ and $\Phi(x_0,\nu)=h(R)D_{\nu\nu}u+B|Du|^2$, we can get
\begin{eqnarray*}
\begin{split}
\Phi(x_0,\zeta_0)&=\langle\zeta_0, \theta  \rangle^2\Phi(x_0,\theta)+\langle \zeta_0, \nu \rangle^2\Phi(x_0,\nu)\\
&\leq \langle\zeta_0,\theta  \rangle^2 \Phi(x_0,\zeta_0)+\langle \zeta_0,\nu \rangle^2\Phi(x_0,\nu),
\end{split}
\end{eqnarray*}
it implies that $\Phi(x_0,\zeta_0)\leq \Phi(x_0,\nu)$ and then
$$D_{\zeta_0\zeta_0}u\leq C(1+\max_{\partial \Omega}|D_{\nu\nu}u|).$$

\textbf{Subcase 2.2: }$\zeta_0$ is a tangential vector of $x_0\in\partial \Omega$, we have
\begin{eqnarray*}
\begin{split}
D_{\zeta_0\zeta_0}D_\nu u&=D_{\zeta_0\zeta_0}(D_ku\nu^k)=D_{\zeta_0}(D_k D_{\zeta_0}u\nu^k+D_k u D_{\zeta_0}\nu^k)\\
&=D_k D_{\zeta_0\zeta_0}u \nu^k+2D_kD_{\zeta_0}u D_{\zeta_0}\nu^k+D_k u D_{\zeta_0\zeta_0}\nu^k\\
&=D_\nu D_{\zeta_0\zeta_0}u+2D_{\zeta_0}D_ku D_{\zeta_0}\nu^k+D_k u D_{\zeta_0\zeta_0}\nu^k
\end{split}
\end{eqnarray*}
then
\begin{eqnarray*}
\begin{split}
0\leq &D_\nu \Phi(x_0,\zeta_0)=D_\nu\bigg(h(D_{\zeta_0\zeta_0}u-a^lD_lu-b)+B|Du|^2\bigg)\\
=&-A(D_{\zeta_0\zeta_0}u-a^lD_lu-b)+D_\nu D_{\zeta_0\zeta_0}u\\
&-D_\nu a^lD_lu-a^lD_\nu D_l u-D_\nu b+2BD_kuD_\nu D_k u\\
\leq &-AD_{\zeta_0\zeta_0}u+D_\nu D_{\zeta_0\zeta_0}u+CB|D_\nu D_ku|+C(1+A)\\
=&-AD_{\zeta_0\zeta_0}u+ D_{\zeta_0\zeta_0}D_\nu u-2D_{\zeta_0}D_kuD_{\zeta_0}\nu^k-D_{\zeta_0\zeta_0}\nu^kD_k u\\
&+CB|D_\nu D_k u|+C(1+A)\\
\leq & (-A+\varphi_{u})D_{\zeta_0\zeta_0}u-2D_{\zeta_0}D_ku D_{\zeta_0}\nu^k+CB|D_\nu D_k u|+C(1+A)\\
\end{split}
\end{eqnarray*}
Note that
\begin{eqnarray*}
\begin{split}
-2D_{\zeta_0}D_ku D_{\zeta_0}\nu^k
\leq2\max_{\partial\Omega} \{|\Pi_{ij}|\}|D_{\zeta_0\zeta_0}u|,
\end{split}
\end{eqnarray*}
 Setting $e_k=a\theta+b\nu$, and $\langle\theta,\nu\rangle=0$, we get
\begin{eqnarray*}
\begin{split}
|D_\nu D_k u|&=|D_\nu Du\cdot e_k|=|D_\nu Du\cdot(a\theta+b\nu)|\\
&=|D_\nu aD_\theta u+aD_\nu D_\theta u+D_\nu bD_\nu u+bD_{\nu\nu}u|\\
&\leq C+C|D_{\nu\nu}u|,
\end{split}
\end{eqnarray*}
so
\begin{eqnarray*}
\begin{split}
0\leq& (-A+\varphi_{u}+2\max_{\partial\Omega} \{|\Pi_{ij}|\})|D_{\zeta_0\zeta_0}u|+CB(1+A+|D_{\nu\nu}u|).
\end{split}
\end{eqnarray*}
Since $-A+\varphi_{u}+2\max_{\partial\Omega} \{|\Pi_{ij}|\}<-1$, then
$$|D_{\zeta_0\zeta_0}u|\leq C(1+|D_{\nu\nu}u|).$$
Thus, we get
\begin{eqnarray*}
\begin{split}
\Phi(x_0,\zeta_0)&=D_{\zeta_0\zeta_0}u+B|Du|^2 \leq C(1+\max_{\partial\Omega}|D_{\nu\nu}u|).
\end{split}
\end{eqnarray*}
where $C$ is a positive constant depending on $n, k, l, p, |u|_{C^1}, |\tilde{\psi}|_{C^2}, |\varphi|_{C^3}$ and $\Omega$.

To sum up, for any $(x,\zeta)\in \Omega\times \mathbb{S}^{n-1}$, we always have
\begin{eqnarray*}
\begin{split}
D_{\zeta\zeta}u&\leq \frac{\Phi(x,\zeta)-B|Du|^2}{h}+v(x,\zeta) \leq C(1+\max_{\partial\Omega}|D_{\nu\nu}u|).
\end{split}
\end{eqnarray*}
\end{proof}

\begin{theorem}\label{C22}
Let~$u\in C^4(\Omega) \cap C^3(\overline\Omega)$~be the $(\Lambda,k)$-convex solution of equation ~\eqref{(1.4)}. Denoting the tangential direction $\tau$ and the outer unit normal $\nu$ at any point $y \in \partial \Omega$,
then
$$
|D_{\tau\nu}u(y)|\leq C,
$$
where $C$ is a positive constant depending on $n, k, l, p, |u|_{C^1}, |\varphi|_{C^1}$ and $\Omega$.
\end{theorem}
\begin{proof}
The idea of the proof is similar to the proof of Lemma 12 in \cite{MQ19}.
\end{proof}

\subsection{ Double normal estimates on the boundary}
\begin{theorem}\label{C23}
Let~$u\in C^4(\Omega) \cap C^3(\overline\Omega)$~be the $(\Lambda,k)$-convex solution of equation ~\eqref{(1.4)}. then
$$
\underset{\partial\Omega}{\max}|D_{\nu\nu}u|\leq C,
$$
where $C$ is a positive constant depending on $n, k, l, p, |u|_{C^1}, |\varphi|_{C^2}, |\tilde{\psi}|_{C^1}$  and $\Omega$.
\end{theorem}

\begin{proof}
Setting $L=\underset{\partial \Omega}{\max} |D_{\nu\nu}u|$, the proof of the theorem is categorized into the following two cases.

\textbf{Case 1:} $\underset{\partial \Omega}{\max} |u_{\nu\nu}|=-\underset{\partial \Omega}{\inf}  u_{\nu\nu}=-u_{\nu\nu}(x_1)=L$.\\
Consider the auxiliary function:
\begin{eqnarray*}
Q=\langle Du, D\bar r \rangle-\varphi(x,u)+L^{-\frac{1}{2}}[\langle Du,D\bar r \rangle-\varphi(x,u)]^2+\frac{1}{2}L\bar r, \quad in~\Omega_\mu.
\end{eqnarray*}
where $\bar r=-dist(\cdot, \partial \Omega)$ is a smooth function with $\bar r|_\Omega<0$, $\bar r|_{\partial \Omega}=0$, $D\bar r|_{\partial \Omega}=\nu$.
$\Omega_\mu:=\{x\in\Omega:d(x,\partial \Omega)<\mu\}$ and $\mu$ is a small positive constant.

We can directly obtain $Q|_{\partial\Omega}=0$. Note that, there exists a positive constant $C_2$ depending on $\mu, |\varphi|_{C^0}, |u|_{C^1}$ and $\partial\Omega$ such that
\begin{eqnarray}\label{431}
Q_{\Omega \backslash {{\Omega}_\mu}}<0, \quad L^{-\frac{1}{2}}\left|[\langle Du,D\bar r \rangle-\varphi(x,u)]\right|\leq \frac{1}{8}.
\end{eqnarray}
when $L> C_2 $. Without loss of generality, we assume that $L> C_2 $, then
we consider two cases for $Q(x_0)=\max \limits_{\bar{\Omega}_\mu} Q $.

\textbf{Subcase 1.1: }$x_0\in \Omega_\mu$. Calculate the partial derivative of $Q $ at $x_0 $
\begin{eqnarray*}
\begin{split}
0=Q_i=&\bigg(1+2L^{-\frac{1}{2}}[\langle Du,D\bar r \rangle-\varphi]\bigg)\bigg(\langle Du,D\bar r \rangle-\varphi\bigg)_i+\frac{1}{2}L \bar r_i
\end{split}
\end{eqnarray*}
we obtain
\begin{eqnarray}\label{432}
\bigg(\langle Du,D\bar r \rangle-\varphi\bigg)_i=-\frac{\frac{1}{2}L \bar r_i}{1+2L^{-\frac{1}{2}}[\langle Du,D\bar r \rangle-\varphi]}.
\end{eqnarray}
Similarly,
\begin{eqnarray}\label{433}
\begin{split}
0\geq& F^{ij}Q_{ij}\\
=&F^{ij}\bigg(1+2L^{-\frac{1}{2}}[\langle Du,D\bar r \rangle-\varphi]\bigg)\bigg(\langle Du,D\bar r \rangle-\varphi\bigg)_{ij}+\frac{1}{2}LF^{ij} \bar r_{ij}\\
&+2F^{ij}L^{-\frac{1}{2}}\bigg(\langle Du,D\bar r \rangle-\varphi\bigg)_i\bigg(\langle Du,D\bar r \rangle-\varphi\bigg)_j\\
=&F^{ij}\bigg(1+2L^{-\frac{1}{2}}[\langle Du,D\bar r \rangle-\varphi]\bigg)\bigg(\langle Du,D\bar r \rangle-\varphi\bigg)_{ij}+\frac{1}{2}LF^{ij} \bar r_{ij}\\
&+\frac{L^{\frac{3}{2}}F^{ij}\bar r_i\bar r_j}{2\bigg(1+2L^{-\frac{1}{2}}[\langle Du,D\bar r \rangle-\varphi]\bigg)^2}
\end{split}
\end{eqnarray}
Since
\begin{eqnarray}\label{434}
\begin{split}
F^{ij}\bigg(\langle Du,D\bar r \rangle-\varphi\bigg)_{ij}
=&F^{ij}(u_{ijl}D_l \bar r+2u_{lj}\bar r_{li}+\bar r_{ijl}D_l u-\varphi_{x_ix_j}\\
&-\varphi_{x_i u}u_j-\varphi_u u_{ij}-\varphi_{ux_j}u_i-\varphi_{uu}u_iu_j)\\
\geq&-C(1+L)\sum_i F^{ii}
\end{split}
\end{eqnarray}
From \eqref{431}, it can be concluded that
\begin{eqnarray}\label{435}
\frac{3}{4}\leq 1+2L^{-\frac{1}{2}}[\langle Du,D\bar r \rangle-\varphi]\leq \frac{5}{4}.
\end{eqnarray}
Combining with \eqref{431}, \eqref{432}, \eqref{433}, \eqref{434},  \eqref{435}, proposition \ref{P4} (\romannumeral3) and proposition \ref{P4} (\romannumeral4), we have
\begin{eqnarray*}
\begin{split}
0\geq&F^{ij}Q_{ij}\\
\geq &-\frac{5}{4}C(1+L)\sum_i F^{ii}+\frac{L^{\frac{3}{2}}F^{ij}\bar r_i\bar r_j}{2\cdot \frac{25}{16}}+\frac{1}{2}F^{ij}L\bar r_{ij}\\
\geq&-C(1+L)\sum_i F^{ii}+C_pL^{\frac{3}{2}}\sum_i F^{ii},
\end{split}
\end{eqnarray*}
hence,
$$ L\leq C.$$
where $C$ is a positive constant depending on $n, k, l, p, |u|_{C^1}, |\varphi|_{C^2}, |\tilde{\psi}|_{C^1}$ and $\Omega$.

\textbf{Subcase 1.2: }$x_0\in \partial\Omega$.
Due to $Q|_{\Omega \backslash {{\Omega}_\mu}}<0$ and $Q|_{\partial \Omega}=0$, then $\underset{\bar{\Omega}}{\max}~Q=Q(x_0)
$. According to Hopf's lemma, at $x_0$ we have
\begin{eqnarray*}
\begin{split}
0\leq& \frac{\partial Q}{\partial \nu}(x_0)\\
=&(D_\nu D_l u D_l \bar r+D_l uD_\nu D_l \bar r-D_\nu \varphi)( 1+2L^{-\frac{1}{2}}[\langle Du,D\bar r \rangle-\varphi])+\frac{1}{2}L\\
\leq&(u_{\nu\nu}+C -D_\nu \varphi)( 1+2L^{-\frac{1}{2}}[\langle Du,D\bar r \rangle-\varphi])+\frac{1}{2}L\\
\leq& -\frac{3}{4}L+\frac{5}{4}C+\frac{1}{2}L+\frac{5}{4}|\varphi_\nu|,
\end{split}
\end{eqnarray*}
we obtain
$$ L\leq C.$$

\textbf{Case 2: }$\underset{\partial \Omega}{\max} |u_{\nu\nu}|=\underset{\partial \Omega}{\sup}~u_{\nu\nu}=u_{\nu\nu}(x_2)=L$.

Similarly, set the following auxiliary function
$$\bar{Q}=\langle Du,D\bar r \rangle-\varphi(x,u)-L^{-\frac{1}{2}}[\langle Du,D\bar r \rangle-\varphi(x,u)]^2-\frac{1}{2}L\bar r,$$
The similar argument works for $\bar{Q}$, we also obtain the conclusion.\\

In summary,
$$\underset{\partial \Omega}{\max} |D_{\nu\nu}u| \leq C.$$
where $C$ is a positive constant depending on $n, k, l, p, |u|_{C^1}, |\varphi|_{C^2}, |\tilde{\psi}|_{C^1}$ and $\Omega$.
\end{proof}

Combining with Theorem \ref{C21}, Theorem \ref{C22} and Theorem \ref{C23},  we have
$$\underset{\bar \Omega}{\sup} |D^2u| \leq C.$$
where~$C$~is a positive constant depending on $n, k, l, p, |u|_{C^1}, |\tilde{\psi}|_{C^2}, |\varphi|_{C^3}$  and $ \Omega$.

\section{Proof of the main theorem}
\begin{proof}
Now we can give the proof of Theorem \ref{T1}. After establishing a priori estimates in Theorem \ref{C0}, Theorem \ref{C11}, Theorem \ref{C12}, Theorem \ref{C21}, Theorem \ref{C22}, Theorem \ref{C23} and Evans-Krylov Theorem, we obtain
$$|u|_{C^{2,\alpha}(\bar \Omega) }\leq C.$$
Applying the method of continuity, we complete the proof of Theorem \ref{T1}.
\end{proof}
\

\end{document}